\documentclass{article}

\usepackage{amsmath}
\usepackage{amssymb}

\input xypic
\input xy
\xyoption{all}
\newtheorem{theorem}{Theorem}

\newtheorem{proposition}[theorem]{Proposition}

\newenvironment{proof}[1][Proof]{\noindent\textbf{#1.} }{\ \rule{0.5em}{0.5em}}

\begin{document}

\author{ Ummahan Ege Arslan\footnote{Corresponding author: Ummahan Ege Arslan, email: uege@ogu.edu.tr
\newline
This work was partially supported by T\"{U}B\.{I}TAK (The Scientific and Technological Research Council Of Turkey) for the second author. \newline
\text{M. S. Classification}:18E20.}\ \ and G\"{u}l\"{u}msen Onarl\i \vspace{0.5cm}\\ Department of
Mathematics and Computer Sciences, \\ Eski\c{s}ehir Osmangazi
University }
\title{\textbf{The Embedding Theorem}}
\date{}
\maketitle

\abstract{In this work, it is shown that the category $\mathsf{XMod/P}$ of crossed modules over fixed group $P$ is an exact category and the complete proof of the embedding theorem of\textsf{\ }$\mathsf{XMod/P}$ into a set valued functor category is given.}

\vspace{0.5cm}

\begin{tabular}{l}
\end{tabular}

\section*{Introduction}
Barr \cite{Barr} introduced exact categories in order to define a good notion
of non-abelian cohomolgy and to generalise Mitchell's \cite{Mitchell} embedding theorem for
Abelian categories to a non-abelian context, in particular; he proved that:

\textquotedblleft every small regular category has an exact embedding into a
set valued functor category (\textsf{C}$^{\text{\textsf{op}}}$,\textsf{Set})
where \textsf{C} is some small category.\textquotedblright

The purpose of this paper is to prove that $\mathsf{XMod/P}$ is an exact category and
the functor
\begin{equation*}
\mathcal{U}:\mathsf{XMod/P}\rightarrow \mathsf{( \Gamma ^{op},Set)}
\end{equation*}%
given by $\mathcal{U}\left( \left( A,\alpha \right) \right) \left(
C(X)\right) =\mathcal{X}\left( \left( C(X),\partial _{X}\right) ,\left(
A,\alpha \right) \right) $ is a full exact embedding functor where $\Gamma ^{op}$ is the opposite full subcategory determined by all finitely generated free crossed $P$-modules.

The proofs we give in more details are inspired by those given for
the corresponding case of algebras in the work of Shammu
\cite{Shammu} and  the proof of embedding theorem is simpler than
that given in Barr \cite{Barr}  for general exact categories.

\section{Crossed Modules}

Crossed modules of groups were initially defined by Whitehead
\cite{Wh1,Wh2} as models for (homotopy) $2$-types. Other related
studies with that concept may be found in \cite{Arslan, Arvasi,
Brown2, PORTER}. We recall from \cite{PORTER}
the definition of crossed modules of groups.

A crossed $P$-module, $(M,P,\partial )$, consists of groups $M$ and $P$ with
a left action of $P$ on $M$, written $\left( p,m\right) \mapsto $ $^{p}m$
and a group homomorphism $\partial :M\rightarrow P$ satisfying the following
conditions:%
\begin{equation*}
CM1)\text{\hspace{0.75cm}}\partial \left( ^{p}m\right) =p\partial \left(
m\right) p^{-1}\text{\qquad\ and \qquad }CM2)\hspace{0.75cm}\text{ }%
^{\partial \left( m\right) }n=mnm^{-1}\text{ }
\end{equation*}%
for $p\in P,m,n\in M$. We say that $\partial :M\rightarrow P$ is a
pre-crossed module, if it satisfies CM$1.$

If $(M,P,\partial )$ and $(M^{\prime },P^{\prime },\partial ^{\prime
})$ are crossed modules, a morphism
\begin{equation*}
\left( \mu ,\eta \right) :(M,P,\partial )\rightarrow (M^{\prime },P^{\prime
},\partial ^{\prime }),
\end{equation*}%
of crossed modules consists of group homomorphisms $\mu :M\rightarrow
M^{\prime }$ and $\eta :P\rightarrow P^{\prime }$ such that%
\begin{equation*}
(i)\text{ }\eta \partial =\partial ^{\prime }\mu \text{ \qquad and \qquad }%
(ii)\text{ }\mu \left( ^{p}m\right) ={^{\eta (p)}\mu \left( m\right)
\text{ }}
\end{equation*}%
for all $p\in P,m\in M.$

Crossed modules and their morphisms form a category, of course. It will
usually be denoted by \textsf{XMod}$.$ We also get obviously a category
\textsf{PXMod} of pre-crossed modules.

There is, for a fixed group $P$, a subcategory \textsf{XMod}$/P$ of \textsf{%
XMod}, which has as objects those crossed modules with $P$ as the
\textquotedblleft base\textquotedblright , i.e., all $(M,P,\partial )$ for
this fixed $P$, and having as morphisms from $\left( M,P,\partial \right) $
to $(M^{\prime },P^{\prime },\partial ^{\prime })$ those $\left( \mu ,\eta
\right) $ in \textsf{XMod} in which $\eta :P\rightarrow P^{\prime }$ is the
identity homomorphism on $P.$ In \textsf{XMod}$/P$, we use $(M,\partial )$ and $\mu$
instead of $(M,P,\partial )$ and $(\mu,id_P )$ respectively. \\

Some standart examples of crossed modules are:

\begin{enumerate}
\item[(i)] A conjugation crossed module is an inclusion of a normal subgroup $%
N\vartriangleleft G$, with action given by conjugation. In
particular, for any group $P$ the identity morphism
$id_P:P\rightarrow P$ is a crossed module with the action $P$ on
itself by conjugation. That is, an inclusion is replaced by a
homomorphism with special properties.

\item[(ii)] If $M$ \ is a group, its automorphism \ crossed module is given by $\chi :M\rightarrow Aut\left( M\right) $ where $\chi \left(
m\right)$ $\ $for $m,n \in M$ is the inner automorphism of $M$
mapping $n$ to $mnm^{-1}.$

\item[(iii)] The trivial crossed module $0:M\rightarrow P$ whenever $M$ is
a $P$-module.

\item[(iv)] A central extension crossed module, i.e. a surjective boundary
$ \mu
:M\rightarrow P $ with the kernel contained in the centre of $M$ and $%
p\in P$ acts on $m\in M$ \ by conjugation with any element of $\mu
^{-1}(p).$

\item[(v)] Any homomorphism $\left( \mu :M\rightarrow P\right) $ with $M$ abelian and $%
Im \mu $ in the centre of $P,$ provides a crossed module with $P$
acting trivially on $M.$
\end{enumerate}

\section{Categorical Structures of Crossed Modules of Groups}

We give some categorical structures of crossed modules of groups that will be used later in proofs. Some of these constructions, namely pullbacks and pushouts of crossed modules of groups may be found in the work of Brown \cite{Brown1}. But we will reconsider and deeply analyse them. Also we investigate the constructions of equaliser and coequaliser of crossed $P$-modules in \textsf{XMod}$/P$.

\begin{proposition} \label{ug3}
In \textsf{XMod}$/P$ every pair of morphisms with common domain and codomain
has an equaliser.
\end{proposition}

\begin{proof}
Let $f,g:$ $\left( C,\partial \right) \rightarrow \left( D,\delta \right) $
be two morphisms of crossed $P$-modules. Let $E$ be the set $E=\{c\in
C:f(c)=g(c)\}.$ It is clear that
\[
\begin{array}{rl}
\varepsilon : & E\rightarrow P \\
& e\mapsto \partial u\left( e\right)
\end{array}%
\]%
is restriction of $\partial$ to $E$ where $u:E \rightarrow C$ is an inclusion. Also $u:\left(
E,\varepsilon \right) \rightarrow $ $\left( C,\partial
\right)$ is a morphism of crossed $P$-modules.

Suppose that there exist a crossed $P$-module, $\left( E^{\prime
},\varepsilon ^{\prime }\right) $ and a morphism $u^{\prime }:\left(
E^{\prime },\varepsilon ^{\prime }\right) \rightarrow $ $\left( C,\partial
\right) $ of crossed $P$-modules such that $fu^{\prime }=gu^{\prime }$. We can define
\[
\begin{array}{rrll}
\gamma : & E^{\prime }& \rightarrow & E \\
         & x          &  \mapsto     & u^{\prime } \left( x\right)
\end{array}%
\]
for all $x\in E^{\prime },$ since $f\left( u^{\prime }x\right) =g\left( u^{\prime
}x\right) ,$ and so $u^{\prime }x\in E.$ The fact that $u^{\prime }$ is a
morphism immediately gives that $\gamma $ is one as well. This morphism is
unique for the diagram
\begin{equation*}
\xymatrix@R=35pt@C=45pt{\ \ E \ar[r]^{u} \ar[dr]|(0.7){^{\varepsilon}} & C \ar[d]%
^\partial \ar@<-0.75ex>[r]_-g \ar@<0.75ex>[r]^-f & D \ar[dl]^-{%
\partial^{\prime}} & \\
\ E^{\prime } \ar[ur]|(0.7){_{u^{\prime}}} \ar@{-->}[u]^-{\gamma}
\ar[r]_-{\ \varepsilon^{\prime }} & P & }
\end{equation*}%
or more simply as
\begin{equation*}
\xymatrix{ (E,\varepsilon ) \ar[r]^{u} & ( C,\partial)
\ar@<-0.75ex>[r]_-g \ar@<0.75ex>[r]^-f & ( D,\partial^{\prime}) & \\
\ (E^{\prime },\varepsilon ^{\prime }) \ar[ur]_{u^{\prime}}
\ar@{-->}[u]^{\gamma} & & }
\end{equation*}%
to commute, i.e. $u\gamma =u^{\prime }.$ Hence $u$ is the equaliser of $%
f$ and $g$.
\end{proof}

\begin{proposition} \label{ug4}
In $\mathsf{XMod}/P$, every pair of morphisms of crossed modules
with common domain and codomain has a coequaliser.
\end{proposition}

\begin{proof}
Let $f,g:(C,\partial )\rightarrow (B,\beta )$ be two morphisms of crossed $P$%
-modules. Let $N$ be the normal subgroup in $B$ generated by all the
elements of the form $f\left( c\right) g\left( c\right) ^{-1}$ for
$c\in C. $ Since $N\subseteq Ker\beta$, we define crossed module
\[
\begin{array}{rl}
\overline{\beta }: & \overline{B} \rightarrow P \\
& Nb\mapsto \beta \left( b\right)
\end{array}%
\]
for each element $b \in B $ where $\overline{B}=B/N.$ By the
definition of $\left( \overline{B},\overline{\beta}\right) $, the
morphism $p:$ $(B,\beta )\rightarrow
\left(\overline{B},\overline{\beta}\right) $ is the induced
projection and it is a morphism of crossed $P$-modules. Suppose
there exist a crossed $P$-module, $(B^{\prime },\beta ^{\prime })$
say, and a morphism of crossed $P$-modules, $p^{\prime }:(B,\beta
)\rightarrow (B^{\prime },\beta ^{\prime })$ such that $p^{\prime
}f=p^{\prime }g,$ then there exists a unique $\varphi :\left( \overline{B},\overline{%
\beta}\right) \rightarrow (B^{\prime },\beta ^{\prime })$ given by
$\varphi \left( Nb\right) =p^{\prime }(b),$ satisfying $\varphi
p=p^{\prime }.$ We will get the following commutative diagram:
\[
\xymatrix@R=35pt@C=45pt{ C \ar@<-0.75ex>[r]_-g \ar@<0.75ex>[r]^-f
\ar[dr]_\partial & B
\ar[d]_\beta \ar[r]^p \ar[dr]|(0.7){^{p^{\prime}}}  & \overline{B} \ar[d]^\varphi \ar[dl]|(0.7){^{\overline\beta}} & \\
\ & P & B^{\prime} \ar[l]^{\beta^{\prime}} & }
\] or more simply as
\[
\xymatrix{ (C,\partial) \ar@<-0.75ex>[r]_-g \ar@<0.75ex>[r]^-f &
(B,\beta) \ar[dr]_{p^{\prime}}
\ar[r]^p & {(\overline{B},%
\overline{\beta})} \ar@{-->}[d]^{\varphi} & \\
\ & & {(B^{\prime},\beta^{\prime}).} }
\]
Therefore $p$ is the coequaliser of $f$ and $g$.
\end{proof}

\begin{proposition} \label{ug2}
The category \textsf{XMod}$/P$ of crossed $P$-modules of groups has
pullbacks.
\end{proposition}

\begin{proof}
Let $f:$ $\left( C,\partial \right) \rightarrow \left( B,\beta \right)$
and $g:$ $\left( D,\delta \right) \rightarrow \left( B,\beta \right) $ be
two morphisms of crossed $P$-modules. Then $\left( C,f\right) $ and $\left(
D,g\right) $ are crossed $B$-modules, where $B$ acts on $C$ and $D$ via $%
\beta $. Form the group
\begin{equation*}
X=C\times _{B}D=\left\{ \left( c,d\right) :f(c)=g(d)\right\} \subseteq
C\times D
\end{equation*}%
Define the morphism $\chi :C\times _{B}D\rightarrow B$ by $\chi \left(
c,d\right) =f(c)=g(d).$ We will show that this is a crossed $B$-module.%
\begin{equation*}
\begin{array}{rcl}
^{\chi (c,d)}(c^{\prime },d^{\prime }) & = & ^{f(c)}(c^{\prime
},d^{\prime })
\\
& = & \left( ^{f(c)}c^{\prime },^{f(c)}d^{\prime }\right)  \\
& = & \left( ^{f(c)}c^{\prime },^{g(d)}d^{\prime }\right)  \\
& = & \left( ^{\beta f(c)}c^{\prime },^{\beta g(d)}d^{\prime }\right)  \\
& = & \left( ^{\partial (c)}c^{\prime },^{\delta (d)}d^{\prime }\right)  \\
& = & \left( cc^{\prime }c^{-1},dd^{\prime }d^{-1}\right)  \\
& = & (c,d)(c^{\prime },d^{\prime })(c,d)^{-1}.%
\end{array}%
\end{equation*}%
Therefore $\left( X,\chi \right) $ is a crossed $B$-module. We can
define
\[
\begin{array}{rl}
\chi ^{\prime }: & X\rightarrow P \\
& (c,d)\mapsto \partial (c)=\delta (d)%
\end{array}%
\]
since $\partial(c)=\beta f(c)=\beta g(d)=\delta (d).$ Crossed module
conditions are immediate. Thus $\left( X,\chi \right) $ is a crossed
$P$-module.

There are two induced morphisms
\begin{equation*}
p:\left( X,\chi \right) \rightarrow \left( C,\partial \right) \qquad
q:\left( X,\chi \right) \rightarrow \left( D,\delta \right)
\end{equation*}%
given by the projections; note that $fp=gq,$ this shows that the diagram%
\begin{equation*}
\xymatrix{ \left( X,\chi \right) \ar[r]^p \ar[d]_q & \left( C,\partial \right) \ar[d]^f & \\ \ \left( D,\delta \right) \ar[r]_g & \left( B,\beta \right) & }
\end{equation*}%
commutes. This construction satisfies a universal property. Therefore
\textsf{XMod}$/P$ has pullbacks.
\end{proof}

\begin{proposition}
The category \textsf{XMod}$/P$ has pushouts.
\end{proposition}

\subsection{Free Crossed Modules}

The definition and construction of a free crossed $P$-module is due
to Whitehead \cite{Wh2}. But we rewrite them from \cite{Brown1}.

Let
$\left(
A,\delta \right) $ be a pre-crossed $P$-module, let $X$ be a set and let $%
\nu :X\rightarrow A$ be a function. We say $\left( A,\delta \right) $ is a
free pre-crossed $P$-module with basis $\nu $ if for any pre-crossed $P$%
-module $\left( A^{\prime },\delta ^{\prime }\right) $ and function $\nu
^{\prime }:X\rightarrow A^{\prime }$ such that $\delta ^{\prime }\nu
^{\prime }=\delta \nu ,$ there is a unique morphism $\phi :\left( A,\delta
\right) \rightarrow \left( A^{\prime },\delta ^{\prime }\right) $ of pre-crossed $P$-modules such that $\phi \nu =\nu ^{\prime }.$ In such case,
we also emphasise the role of the function $\omega=\delta \nu :X\rightarrow P$ by
calling $\left( A,\delta \right) $ with the function $\nu ,$ a free
pre-crossed $P$-module on $\omega$.

If $\left(
A,\delta \right) $ is a crossed $P$-module and $\left(
A,\delta \right) $ with $\nu$ has the above universal property for morphisms into crossed $P$-modules, then we call $\left(
A,\delta \right) $ a free crossed $P$-module with basis $\nu$ (or on $\omega$).
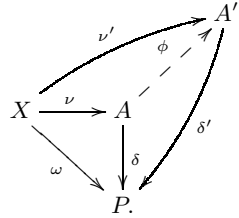
\[
\xymatrix{ & & A^{\prime} \ar@/^/[ddl]^{\delta ^{\prime }} & \\
\ X \ar[r]^\nu \ar[dr]_{\omega} \ar@/^/[urr]^{\nu ^{\prime }} & A \ar[d]^\delta \ar@{-->}[ur]^-{\phi} & & \\
\ &  P. & & }
\]

\begin{proposition}
Let $P$ be a group, $X$ a set and $\omega:X\rightarrow P$ a function.
Then a free pre-crossed $P$-module on $\omega$, and a free crossed
$P$-module on  $\omega$, exist, and are each uniquely determined up to
isomorphism.
\end{proposition}

\begin{proof}
We let $H$ be the free group on the set $X\times P,$ and we write
the elements of this set as $^{u}\rho ,\rho \in X,u\in P.$ Let $P$
acts on $H$ by acting on the generators as $^{v}\left( ^{u}\rho
\right) ={^{v}\rho} ,\rho \in X$ and $u,v \in P.$ Define $\theta
:H\rightarrow P$ by its values on
the generators%
\[
\theta \left( ^{u}\rho \right) =u\omega \left( \rho\right) u^{-1}
\]%
$\rho \in X,u\in P.$ Define $\nu :X\rightarrow H$ by $\nu \left(
\rho \right) =\rho ^{1},\rho \in P,$ so that $\theta \nu =\omega .$
Then $\left( H,\theta \right) $ is a pre-crossed $P$-module, which,
with $\nu ,$ is easily checked to be a free pre-crossed $P$-module
on $\omega .$

From $\left( H,\theta \right) $ we can form a crossed $P$-module $\left(
C,\partial \right) $ by factoring out the Peiffer elements. Then $\left(
C,\partial \right) $ with the composite $X\rightarrow H\rightarrow C$ is a
free crossed $P$-module on $\omega$.

\quad The uniqueness of these constructions up to isomorphism follows by
the\ usual universal argument.
\end{proof}

\subsection{\bigskip Singly Generated Free Crossed Modules}

Singly generated free crossed modules will have important role in the embedding theorem.

Suppose that for each $x\in P$, we take the symbol $\overline{x}$ and form the singly generated free crossed $P$%
-module $(C(\overline{x}),\partial _{x})$ on $\left( \{\overline{x}\},\omega
_{x}\right) $, where $\omega _{x}$ takes $\overline{x}\mapsto x\in P$ and $%
\partial _{x}(\overline{x})=\omega _{x}\left( \overline{x}\right) =x$.
We note that the Peiffer normal subgroup $P(\overline{x})$ is
generated by all the elements of the form $\{\overline{x}$
$\overline{x}$ $\overline{x}^{-1}(^{x}\overline{x})^{-1}\mid x\in
P\}.$ Thus each element in $C(\overline{x})$ will have the form
\[
c=P(\overline{x}){^p\overline{x}}\text{ \ for some }p\in P.
\]

\section {$\mathsf{XMod}$/P is an Exact Category}

Barr \cite{Barr} defined exact categories and used them to generalise embedding theorem the sense of Mitchell's to a non-abelian context.
We will give the definition and examples of exact category having appeared in Barr \cite{Barr}.

Let \textsf{X} be a category. We say \textsf{X }is regular if it
satisfies $EX1)$ below and exact if it satisfies $EX2)$ in addition.

$EX1)$ The kernel pair of every morphism exists and has a coequaliser.

$EX2)$ Every equivalence relation is effective that is every equivalence relation is a
kernel pair.

Examples of exact categories:
\begin{enumerate}
\item[(i)] The category \textsf{Set }of sets.

\item[(ii)] The category of non empty sets.

\item[(iii)]Any abelian category.

\item[(iv)]Every partially ordered set considered as a category.

\item[(v)] For any small category \textsf{C}, the functor category (\textsf{C}$^{%
\text{\textsf{op}}}$,\textsf{Set}).
\end{enumerate}

\begin{proposition}
In \textsf{XMod}$/P$ every morphism has a kernel pair and the kernel
pair has a coequaliser.
\end{proposition}

\begin{proof}
Let $\ f:\left( A,\alpha \right) \rightarrow \left( B,\beta \right) $ \ be a
morphism of crossed $P$-modules. Since \textsf{XMod}$/P$ has pullbacks by proposition \ref{ug2}, then
the pullback diagram of the pair equal morphisms $f:A\rightarrow B$, is the
diagram
\begin{equation*}
\xymatrix { A\times _{B}A \ar[r]^-{p_1} \ar[d]_{p_2} & A \ar[d]^f \\ \
A \ar[r]_f & B }
\end{equation*}%
where $A\times _{B}A=\left\{ \left( a,a^{\prime }\right) \mid f(a)=f(a^{\prime
})\right\}$.

Then $(p_{1},p_{2})$ is the kernel pair of the morphisms $f$ where
$p_{1}(a,a^{\prime })=a$ and $p_{2}(a,a^{\prime })=a^{\prime }$.
Since every pair of morphisms of crossed modules with common domain
and codomain has a coequaliser in \textsf{XMod}$/P$ by proposition
\ref{ug4}, the pair $(p_{1},p_{2})$ has a coequaliser.
\end{proof}

Now we will prove that \textsf{XMod}$/P$ satisfies axiom EX2 above.

\begin{proposition}
Every equivalence relation $(E,\varepsilon) \overset{u}{\underset{v}{\rightrightarrows }}(A,\alpha)$ in \textsf{XMod}$/P$ is effective.
\end{proposition}

\begin{proof}
Let $A/E$ be the set of all equivalence classes $[a]_{E}$ with respect to $E$%
, where $[a]_{E}=\{b\in A:(a,b)\in E\}.$ It has the structure of a group with the following multiplication:
\begin{equation*}
\lbrack a]_{E}[b]_{E}=[ab]_{E}
\end{equation*}%
for $[a]_{E},[b]_{E}\in A/E$.

The multiplication is independent of the choice of the
representatives, $a,b\in A.$ The morphism $\alpha:A \rightarrow P $
will induce a morphism $\overline{\alpha}:A/E\rightarrow P,$ given
by $[a]_{E}\mapsto \alpha (a).$ Since $E$ is the subobject of
$A\times _{P}A,$ by the definition of an
equivalence relation on $A$, i.e.%
\begin{equation*}
E\subseteq A\times _{P}A=\{\left( a,a^{\prime }\right) :\alpha (a)=\alpha
(a^{\prime })\}
\end{equation*}%
then, $\left( u(x),v(x)\right) \in E$ for $x\in E,$ since $\alpha
u(x)=\alpha v(x),$ thus $\left( 1,u(x)v(x)^{-1}\right) \in E,$ so
$[u(x)v(x)^{-1}]_E=[1]_E$. Therefore $\alpha u=\alpha v$ and
$\overline{\alpha}$ is well defined. $( A/E,\overline{\alpha})$ is a
crossed $P$-module with the action defined by
$^{p}[a]_{E}=[^{p}a]_{E}$ for $[a]_{E}\in A/E$ and $p\in P$. The
verification of the axiom CM$1$ is immediate, while CM$2$ is proved
as follows:
\begin{equation*}
^{\overline{\alpha}([a]_{E})}[b]_{E}={^{\alpha
(a)}[b]_{E}}=[^{\alpha
(a)}b]_{E}=[aba^{-1}]_{E}=[a]_{E}[b]_{E}[a]_{E}^{-1}
\end{equation*}%
for $[a]_{E}$ and $[b]_{E}\in A/E$ . Thus we will get the
following diagram
\begin{equation*}
\xymatrix { E \ar@<-0.75ex>[r]_-v \ar@<0.75ex>[r]^-u
\ar[dr]_{\varepsilon} & A \ar[r]^-p \ar[d]^-\alpha & A/E
\ar[dl]^-{\overline{\alpha}} & \\ \  & P & & }
\end{equation*}%
of morphisms of crossed $P$-modules, where $p$ is the projection into $A/E.$
Again since $\alpha u=\alpha v,
pu=pv.$ Suppose that there exist an object $\left( D,\delta \right) \in
\mathcal{X}$ with $u^{\prime },v^{\prime }:D\rightarrow A,$ such that $%
pu^{\prime }=pv^{\prime },$ i.e. $pu^{\prime }(d)=pv^{\prime }(d),$ for all $%
d\in D$ so $[u^{\prime }(d)]_{E}=[v^{\prime }(d)]_{E},$ i.e. $\left(
u^{\prime }(d),v^{\prime }(d)\right) \in E,$ and therefore there exists a
unique morphism $\theta :D\rightarrow E,$ such that $u\theta =u^{\prime }$
and $v\theta =v^{\prime }.$
\[
\xymatrix{ (D,\delta) \ar@/^/[drr]^{u^{\prime}} \ar@/_/[ddr]_{v^{\prime}} \ar@{-->}[dr]^\theta & & & \\
\ & (E,\varepsilon) \ar[r]^u \ar[d]_v &  (A,\alpha) \ar[d]^p & \\
\ & (A,\alpha) \ar[r]^p & (A/E,\overline{\alpha}). & }
\]
Thus $\left( u,v\right) $ is kernel pair of a morphism,
$p:A\rightarrow A/E,$ in \textsf{XMod}$/P$.
\end{proof}

Thus we have already proved the following theorem.
\begin{theorem}
The category \textsf{XMod}$/P$ is an exact category.
\end{theorem}

\begin{proposition}
The category \textsf{XMod}$/P$ has a set of generators.
\end{proposition}

\begin{proof}
Let $\Gamma $ be the full subcategory of \textsf{XMod}$/P$
determined by all
finitely generated free crossed $P$-modules, i.e.%
\[
\Gamma =\left\{ \left( C(X),\partial _{X}\right) :C(X)\text{ free on }%
w_{X}:X\rightarrow P\text{ and }X\text{ finite set}\right\} .
\]%
Let $\ m:\left( A,\alpha \right) \rightarrow \left( B,\beta \right)$ be
any monomorphism of crossed $P$-modules which is not an isomorphism, then, $%
\beta m=\alpha $ and so $Im\alpha \subseteq Im\beta .$ Since $m$ is a monomorphism and not an
isomorphism, we can choose $b\in B$
such that $b\notin Im$ $m,$
and let $y=\beta (b).$ Form the singly generated free crossed $P$-module $%
\left( C(\overline{y}),\partial _{y}\right) $ on $\left( \left\{
\overline{y}\right\}
,w_{y}\right) ,$ where $w_{y}\left( \overline{y}\right) =y,$ and $\partial _{y}(%
\overline{y})=w_{y}\left( \overline{y}\right) =y$ in $P.$ Define the morphism $g:C(%
\overline{y})\rightarrow B$ on the generator, by
$g(\overline{y})=b.$

Suppose $g$ factors through a morphism $h:C(\overline{y})\rightarrow
A,$ such that $\ mh=g.$ Thus $b\in Im\ m.$ Hence the morphism $g$
does not factor through $m.$
\end{proof}

\begin{proposition}
For $\left( C(\overline{x}),\partial _{x}\right) $ and $\left( C(\overline{y}%
),\partial _{y}\right) $ in $\Gamma ,$ $f:(C(\overline{x}),\partial_{x}) \rightarrow (C(%
\overline{y}),\partial_{y}) $ exists if and only if $x=pyp^{-1},$
for some $p\in P.$
\end{proposition}

\begin{proof}
If $f$ exists, then $\partial _{y}f=\partial _{x}$ and thus
$Im\partial _{y}\supseteq Im\partial _{x},$ i.e. $x\in \left(
y\right) $, so $x=pyp^{-1}$ for some $p\in P.$

On the other hand, suppose that $x=pyp^{-1},$ for some $p\in P,$ define $f:C(%
\overline{x})\rightarrow C(\overline{y})$ by $f({^{t}\overline{x}})=f({^{t}%
\overline{pyp^{-1}}})={^{tp}\overline{y}}$ for all $t,p\in P.$ To
check that $f$ is a crossed module morphism, let
${^{t}\overline{x}},{^{t^{\prime}}\overline{x}} \in C(\overline{x})$
then
\[
\begin{array}{rcl}
\partial _{y}f\left( {^{t}\overline{x}}\right)  & = & \partial _{y}f(^{t}%
\overline{pyp^{-1})} \\
& = & \partial _{y}\left( ^{t}{}^{p}\overline{y}\right)  \\
& = & t\partial _{y}\left( {}^{p}\overline{y}\right) t^{-1} \\
& = & tp\partial _{y}\left( \bar{y}\right) p^{-1}t^{-1} \\
& = & tpyp^{-1}t^{-1} \\
& = & txt^{-1} \\
& = & t\partial _{x}(\bar{x})t^{-1} \\
& = & \partial _{x}\left( {^{t}\overline{x}}\right)
\end{array}%
\]
also for all $q\in P,$%
\[
\begin{array}{rcl}
{^{q}f({^{t}\overline{x})}} & = & {^{q}f(^{t}\overline{pyp^{-1}})} \\
& = & {^{q}(^{tp}\overline{y})} \\
& = & ({^{qtp}\overline{y}}) \\
& = & f({^{qt}{\overline{pyp^{-1}}}}) \\
& = & f({^{q}({^{t}\overline{pyp^{-1}}})}) \\
& = & f({^{q}({^{t}\overline{x}})}).%
\end{array}%
\]%
Thus $f$ is a crossed module morphism.
\end{proof}

\section{The Embedding Theorem}

We give the proof of the full embedding theory of the category $\mathsf{XMod/P}$ of crossed modules into the functor category $\left( \mathsf{\Gamma ^{op},Set}\right),$
where $\textsf{Set}$ is the category of sets and $\Gamma$ is the full subcategory of $\mathsf{XMod/P}$ determined by all finitely generated free crossed $P$-modules.

We will start giving the notion of \textquotedblleft labelling\textquotedblright, we need the proof of the Embedding Theorem.

Let $\ \Gamma $ be as above. Note that for each $a\in A$ there is a free
crossed $P$-module, $(C(\overline{x}),\partial _{x})$ and a unique morphism,
$<a>:(C(\overline{x}),\partial _{x})\rightarrow \left( A,\alpha \right) ,$
where $x=\alpha \left( a\right) ,$ which takes $\overline{x}\mapsto a\in A$
, we will call this morphism, the labelling of $a,$ so $\mathsf{XMod/P}%
\left( (C(\overline{x}),\partial _{x}),\left( A,\alpha \right) \right) \cong
\alpha ^{-1}\left( x\right) .$ Thus for $a,a^{\prime }\in A,$ we will have
labellings, \ $<aa^{\prime }>,$ for $\alpha \left( a\right) =x$ and $\alpha
\left( a^{\prime }\right) =y$ in $P.$

We will write $\mathcal{X}=\mathsf{XMod}/P$ in the rest of this article.
\newpage
\begin{theorem}
(The Embedding Theorem) \label{ug1}
\end{theorem}

The functor
\begin{equation*}
\mathcal{U}:\mathsf{X}\rightarrow \mathsf{( \Gamma ^{op},Set)}
\end{equation*}%
given by $\mathcal{U}\left( \left( A,\alpha \right) \right) \left(
C(X)\right) =\mathcal{X}\left( \left( C(X),\partial _{X}\right) ,\left(
A,\alpha \right) \right) $ is a full exact embedding functor.

\begin{proposition}
$\mathcal{U}$ is full.
\end{proposition}

\begin{proof}
Let $\left( A,\alpha \right) $ and $\left( B,\beta \right) $ be two crossed $%
P$-modules in $\mathcal{X}$. Write $\mathcal{U}A$ for
$\mathcal{U}\left(
A,\alpha \right) :\Gamma ^{op}\rightarrow Set.$ Suppose that $\phi :\mathcal{%
U}A\mathcal{\rightarrow U}B$ is a natural transformation. Write
\[
\phi _{\overline{x}}=\phi (C(\overline{x})):\mathcal{U}A(C(\overline{x}))%
\mathcal{\rightarrow U}B(C(\overline{x})).
\]%
Given $a\in A$ we get a unique morphism of crossed $P$-modules
$<a>\in
\mathcal{U}A(C(\overline{x})),$ where $\alpha (a)=x\in P,$ so define a morphism $%
f:\left( A,\alpha \right) $ $\rightarrow $ $\left( B,\beta \right) $ of $\mathcal{X}$ by $%
f(a)=\left( \phi _{\overline{x}}<a>\right) \left(
\overline{x}\right)$ such that $\mathcal{U}f=\phi.$ We need to prove
that $f(aa^{\prime })=f\left( a\right) f\left( a^{\prime }\right) ,$
and $f(^{p}a)={^{p}f(a)}$ for all $a,a^{\prime }\in A$ and $p
\in P$ to check that the morphism $f$ is a morphism of crossed $P$-modules in $%
\mathcal{X}$.

Let $<a>:C(\overline{x})\rightarrow A,<a^{\prime }>:C(\overline{y}%
)\rightarrow A$ and $<aa^{\prime }>:C(\overline{xy})\rightarrow A,$
be the morphisms which are the \textquotedblleft
labels\textquotedblright \ of $a,a^{\prime }, (aa^{\prime }) \in A$
respectively. Thus we can write
\[
f(aa^{\prime })=\phi _{\left( \overline{xy}\right) }\left(
<aa^{\prime }>\right) \left( \overline{xy}\right) .
\]%
Also we have a morphism $\sigma :C(\overline{xy})\rightarrow C(\overline{x},%
\overline{y}),$ given by
$\sigma \left( \overline{xy}\right) =\overline{x} \ \overline{y}$
and hence we get the commutative triangle
\[
\xymatrix { C(\overline{xy}) \ar[r]^\sigma \ar[dr]_{<aa^{\prime }>}
& C(\overline{x},\overline{y}) \ar[d]^{<a><a^{\prime }>} & \\ \ & A
& }
\]%
of morphisms of crossed $P$-modules, where $<a><a^{\prime }>$ is the
unique
morphism defined by:%
\[
\left( <a><a^{\prime }>\right) \left( \overline{x}\right) =a\qquad \text{and}%
\qquad \left( <a><a^{\prime }>\right) \left( \overline{y}\right)
=a^{\prime }.
\]%
Apply $\phi :\mathcal{U}A\mathcal{\rightarrow U}B$ to $\sigma $ we
will get the commutative square
\[
\xymatrix@R=30pt@C=40pt{ \mathcal{U}A\left(
C({\overline{x}},{\overline{y}})\right) \ar[d]_{\mathcal{U}A\left(
\sigma \right)} \ar[r]^-{ \phi( {\overline{x}},{\overline{y}})} &
\mathcal{U}B\left( C({\overline{x}},{\overline{y}})\right)
\ar[d]^{\mathcal{U}B\left( \sigma \right)} & \\ \ \mathcal{U}A\left(
C(\overline{xy})\right) \ar[r]_-{\phi(\overline{xy})} &
\mathcal{U}B\left( C(\overline{xy})\right) & }
\]%
where
\[
\mathcal{U}A\left( \sigma \right) \left( <a><a^{\prime }>\right)
=\left( <a><a^{\prime }>\right) \circ \sigma =<aa^{\prime }>
\]%
and
\[
\left( <a><a^{\prime }>\right)
:C(\overline{x},\overline{y})\rightarrow A,
\]%
given by $\left( \overline{x}\ \overline{y}\right) \mapsto <a>\left( \overline{%
x}\right) <a^{\prime }>\left( \overline{y}\right) .$ Thus
\[
\begin{array}{rcl}
f(aa^{\prime }) & = & \phi _{\left( \overline{xy}\right)
}(<aa^{\prime
}>)\left( \overline{xy}\right)  \\
& = & \phi _{\left( \overline{xy}\right) }\left( (<a><a^{\prime
}>)\circ
\sigma \right) \left( \overline{xy}\right)  \\
& = & \phi _{\left( \overline{xy}\right) }\mathcal{U}A\left( \sigma
\right)
\left( (<a><a^{\prime }>)\right) \left( \overline{x}\allowbreak \;\overline{y%
}\right)  \\
& = & \mathcal{U}B\left( \sigma \right) \circ \phi _{(\overline{x},\overline{%
y})}(<a><a^{\prime }>)\left( \overline{x}\allowbreak
\;\overline{y}\right).
\end{array}%
\]%
But since $\mathcal{X}\left( C(\overline{x},\overline{y}),A\right)
\cong
\mathcal{X}\left( C(\overline{x}),A\right) \times \mathcal{X}\left( C(%
\overline{y}),A\right) ,$ thus
\[
\phi _{(\overline{x},\overline{y})}:\mathcal{X}\left( C(\overline{x},%
\overline{y}),A\right) \rightarrow \mathcal{X}\left( C(\overline{x},%
\overline{y}),B\right)
\]%
is the same up to unique natural isomorphism as%
\[
\mathcal{X}\left( C(\overline{x}),A\right) \times \mathcal{X}\left( C(%
\overline{y}),A\right) \rightarrow \mathcal{X}\left( C(\overline{x}%
),B\right) \times \mathcal{X}\left( C(\overline{y}),B\right) ,
\]%
and $\phi _{(\overline{x},\overline{y})}\left( <a>,<a^{\prime
}>\right) =\left( \phi _{\overline{x}}\left( <a>\right) ,\phi
_{\overline{y}}\left( <a^{\prime }>\right) \right) .$ \\ Now using
the diagram
\[
\xymatrix@R=10pt@C=15pt{ \mathcal{X}\left( C\left( \overline{x}\right) ,A\right)
\times \mathcal{X}\left( C\left( \overline{y}\right) ,A\right)
\ar[rr] \ar[dd] & & \mathcal{X}\left( C\left( \overline{x}\right)
,B\right) \times \mathcal{X}\left( C\left( \overline{y}\right)
,B\right) \ar[dd] & \\ \ & & &
\\ \ \mathcal{U}A({C\left( \overline{x},\overline{y}\right)})
\ar[rr]^{\phi_{\left( \overline{x},\overline{y}\right) }}
\ar[dd]_{\mathcal{U}A(\sigma)} & & \mathcal{U}B({C\left(
\overline{x},\overline{y}\right)}) \ar[dd]^{\mathcal{U}B(\sigma)} &
\\ \ & & & \\ \ \mathcal{U}A({C\left( \overline{xy}\right)})
\ar[rr]_{\phi_{\left( \overline{xy}\right) }} & &
\mathcal{U}B({C\left( \overline{xy}\right)}) & }
\]%
\begin{equation*}
\begin{array}{rcl}
\mathcal{U}B\left( \sigma \right) \circ \phi _{\left( \overline{x},\overline{%
y}\right) }\left( <a>,<a^{\prime }>\right) \left( \overline{x},\overline{y}%
\right)  & = & \mathcal{U}B\left( \sigma \right) \left( \phi _{\overline{x}%
}\left( <a>\right) ,\phi _{\overline{y}}\left( <a^{\prime }>\right) \right)
\left( \overline{x},\overline{y}\right)  \\
& = & \left( \left( \phi _{\overline{x}}\left( <a>\right) ,\phi _{\overline{y%
}}\left( <a^{\prime }>\right) \right) \circ \sigma \right) \left( \overline{%
xy}\right)  \\
& = & \left( \phi _{\overline{x}}\left( <a>\right) ,\phi _{\overline{y}%
}\left( <a^{\prime }>\right) \right) (\overline{x}\text{ }\overline{y}) \\
& = & \left( \phi _{\overline{x}}\left( <a>\right) \phi _{\overline{y}%
}\left( <a^{\prime }>\right) \right) (\overline{x}\text{ }\overline{y}) \\
& = & \phi _{\overline{x}}\left( <a>\right) \left( \overline{x}\right) \phi
_{\overline{y}}\left( <a^{\prime }>\right) (\text{ }\overline{y}) \\
& = & f(a)f(a^{\prime }).%
\end{array}%
\end{equation*}

Thus $f(aa^{\prime })=f(a)f(a^{\prime }).$

We will next prove that $f(^{p}a)={^{p}f(a)}.$ As before suppose $%
\alpha \left( a\right) =x$ and let $y=pxp^{-1}$ and
$<{^{p}a}>:C\left( \overline{y}\right)
\rightarrow A,$ which labels $^{p}a\in A.$ Let $m_{p}^{x}:C\left( \overline{y%
}\right) \rightarrow C\left( \overline{x}\right) $ be the morphism $%
m_{p}^{x}\left( \overline{pxp^{-1}}\right) =$ $^{\text{ }p}%
\overline{x}$, thus $%
<{^{p}a}>= \ <a>\circ m_{p}^{x}$ and \ $\mathcal{U}A\left( m_{p}^{x}\right) :%
\mathcal{U}A\left( C\left( \overline{x}\right) \right) \rightarrow \mathcal{U%
}A\left( C\left( \overline{y}\right) \right) $ will send $<a>$ into $<$ $%
^{p}a>,$ i.e. $\mathcal{U}A\left( m_{p}^{x}\right) \left( <a>\right) =<{^{p}a}>.$ Now apply $\phi :\mathcal{U}A\rightarrow \mathcal{U}B$ to $%
m_{p}^{x},$ this gives a commutative square
\[
\xymatrix@R=20pt@C=30pt{ \mathcal{U}A\left( C\left( {\overline{x}}\right) \right)
\ar[d]_-{\mathcal{U}A\left( m_{p}^{x}\right)} \ar[rr]^-{\phi
_{{\overline{x}}}} & & \mathcal{U}B\left( C\left(
{\overline{x}}\right) \right) \ar[d]^-{\mathcal{U}B\left(
m_{p}^{x}\right)} & \\ \ \mathcal{U}A\left( C\left(
{\overline{y}}\right) \right) \ar[rr]_-{\phi _{{\overline{y}}}} & &
\mathcal{U}B\left( C\left( {\overline{y}}\right) \right) & }
\]%
$f(^{p}a)=\phi _{\overline{y}}\left( <^{p}a>\right) \left( \overline{y}%
\right) ={\phi _{\overline{y}}}\circ \ \mathcal{U}A\left( m_{p}^{x}\right)
\left( <a>\right) \left( \overline{x}\right) =\mathcal{U}B\left(
m_{p}^{x}\right) \circ \ {\phi _{\overline{x}}}\left( <a>\right) \left(
\overline{x}\right) .$ \newline

Now given $<b>=\phi \overline{_{x}}\left( <a>\right) :C\left( \overline{x}\right) \rightarrow B,$
\[
\mathcal{U}B\left( m_{p}^{x}\right) \left( <b>\right) =<b>\circ \ m_{p}^{x}%
\text{ \ \ and \ \ }<b>\circ \ m_{p}^{x}\left( \overline{y}\right)
=<b>\left( ^{p}\overline{x}\right) ={^{p}<b>\left(
\overline{x}\right)} ,
\]
so
\[
\mathcal{U}B\left( m_{p}^{x}\right) \circ \phi
_{\overline{x}}\left( <a>\right) ={ ^{p}f\left( a\right)}.
\]
Also,
\[
\beta f(a)=\beta (\phi _{\overline{x}}<a>)\left( \overline{x}\right)
=\partial _{x}(\overline{x})=x=\alpha \left( a\right).
\]

Hence $f$ is a morphism of crossed $P$-modules in $\mathcal{X}$ and $ \mathcal{U}f=\phi $. Therefore $\mathcal{U}$ is full.
\end{proof}

\begin{proposition}
$\mathcal{U}$ is faithful.
\end{proposition}

\begin{proof}
We can generalise Schubert's \cite{Schubert} definition of a generator set
for the category with pushouts and equalisers to $\mathcal{X}.$ We get a
morphism $h\in \mathcal{X}\left( \left( C(\overline{x}),\partial _{x}\right)
,\left( A,\alpha \right) \right) $ such that $f_{1}h\neq f_{2}h$ for any
couple $\left( \left( A,\alpha \right) ,\left( B,\beta \right) \right) $ of
objects from $\mathcal{X}$, for any two distinct morphisms $f_{1},f_{2}\in
\mathcal{X}\left( \left( A,\alpha \right) ,\left( B,\beta \right) \right) ,$
and for an object $(C(\overline{x}),\partial _{x})$ of the category
determinated by all finitely generated free $P$-modules. Then $\mathcal{U}%
f_{1}\left( h\right) \neq \mathcal{U}f_{2}\left( h\right) $. Thus $\mathcal{U%
}$ is faithful.
\end{proof}

\begin{proposition}
$\mathcal{U}$ is an exact functor.
\end{proposition}

\begin{proof}
By Barr \cite{Barr}, a functor is exact if it preserves all finite
limits and regular epimorphisms. First we will show that
$\mathcal{U}$ preserves regular epimorphisms. Let $p:\left( A,\alpha
\right) \rightarrow \left( \overline{A},\overline{\alpha }\right) $
be the coequaliser of the morphisms $f$, $g:\left( C,\partial
\right) \rightarrow \left( A,\alpha \right) $ of crossed $P$-modules
in $\mathcal{X}$\ constructed in proposition \ref{ug4}. Apply the
functor $\mathcal{U}$ on the coequaliser sequence
\begin{equation*}
\xymatrix { C \ar@<-0.75ex>[r]_-g \ar@<0.75ex>[r]^-f \ar[dr]_\partial  &
A \ar[d]^\alpha \ar[r]^p & \overline{A} \ar[dl]^{\overline{\alpha }} &
\\ \ &  P & & }
\end{equation*}%
we will get the following sequence
\begin{equation*}
\xymatrix{ \mathcal{U}C \ar@<-0.75ex>[r]_{\mathcal{U}g}
\ar@<0.75ex>[r]^{\mathcal{U}f} & \mathcal{U}A \ar[r]^{\mathcal{U}p} &
\mathcal{U}\overline{A} & \\ }
\end{equation*}%
with $\mathcal{U}p\circ \mathcal{U}f=\mathcal{U}p\circ \mathcal{U}g.$ We
will prove that the above diagram is a coequaliser diagram for each $%
(C(X),\partial _{X}).$ Note that for any $h:(C(X),\partial _{X})\rightarrow
\left( C,\partial \right) $ we can say that $Imh$ is the $%
\left( X,w_{X}\right) $- indexed family of elements of $C,$ i.e.
we can write $h(X)=\left\{ c_{x}:h(x)=c_{x}\text{ for }x\in X\right\} .$
Thus, if $c_{x}\in Imh \subseteq C$ such that $%
f\left( c_{x}\right) =a_{x}$ and $g\left( c_{x}\right) =a_{x}^{\prime }$ in $%
A,$ then $Ia_{x}=Ia_{x}^{\prime }$ in $\overline{A}=A/I,$ where $I$ is the
ideal generated by all elements $f\left( c_{x}\right) g\left( c_{x}\right)
^{-1}$ for all $c_{x}\in Imh.$ Now suppose there is a morphism $%
p^{\prime }:\mathcal{X}((C(X),\partial _{X}),\left( A,\alpha \right)
)\rightarrow Y$ such that $p^{\prime }\circ \ \mathcal{U}f=p^{\prime
}\circ \
\mathcal{U}g,$ so we define a morphism $p^{\prime \prime }:\mathcal{X}%
((C(X),\partial _{X}),\left( \overline{A},\overline{\alpha }\right)
)\rightarrow Y$ by $p^{\prime \prime }\left( Ia_{x}\right) =p^{\prime
}\left( a_{x}\right).$ Clearly $p^{\prime \prime }\circ \ \mathcal{U}%
p=p^{\prime }.$ Therefore $\mathcal{U}p$ is the coequaliser of $\mathcal{U}f$
and $\mathcal{U}g$. Hence $\mathcal{U}$ preserves coequalisers.

We will show that $\mathcal{U}$ preserves all finite limits. It is enough to
prove it preserves every finite product and equalisers in $\mathcal{X}.$ Let
$\left( A,\alpha \right) ,\left( B,\beta \right) $ be two crossed $P$%
-modules in $\mathcal{X}$. Then
\begin{equation*}
\begin{array}{rcl}
\mathcal{U(}A\times _{P}B\mathcal{)}\left( C(X)\right) & = & \mathcal{X}%
\left( (C(X),\partial _{X}),\left( A\times _{P}B,\tau \right) \right) \\
& \cong & \mathcal{X}\left( (C(X),\partial _{X}),\left( A,\alpha \right)
\right) \times \mathcal{X}\left( (C(X),\partial _{X}),\left( B,\beta \right)
\right) \\
& \cong & \mathcal{U}\left( A\right) (C(X),\partial _{X})\times \mathcal{U}%
\left( B\right) (C(X),\partial _{X})%
\end{array}%
\end{equation*}%
where the product $(A\times _{P}B,\tau)$ is the pullback over the terminal object $(P,id_P)$.

Thus by induction, $\mathcal{U}$ will preserve finite products.
Finally we will prove that $\mathcal{U}$ preserves equalisers. Let $%
u:E\rightarrow A$ be the equaliser of $f,g:A\rightarrow B$ as proposition \ref{ug3}. Apply $%
\mathcal{U}$ on the sequence
\begin{equation*}
\xymatrix { E \ar[r]^u \ar[dr]_\varepsilon &  A \ar[d]^\alpha
\ar@<-0.75ex>[r]_g \ar@<0.75ex>[r]^f &  B \ar[dl]^\beta & \\ \ &  P & &
& }
\end{equation*}%
we will get \
$$
\mathcal{U}E\overset{\mathcal{U}u}{\mathcal{\longrightarrow }}%
\ \mathcal{UA}\underset{\mathcal{U}g}{\overset{\mathcal{U}f}{\rightrightarrows
}}\mathcal{U}B,
$$
and the
sequence
\begin{equation*}
\mathcal{X(}(C(X),\partial _{X}),\left( E,\varepsilon \right) \mathcal{%
)\rightarrow X(}(C(X),\partial _{X}),\left( A,\alpha \right) \mathcal{%
)\rightrightarrows X(}(C(X),\partial _{X}),\left( B,\beta \right))
\end{equation*}%
for each $(C(X),\partial _{X})\in \Gamma.$ Also
\[
\mathcal{U}f\circ \mathcal{U}u=\mathcal{U}g\circ \mathcal{U}u.
\]

Any $h:(C(X),\partial _{X})\rightarrow \left( E,\varepsilon \right) $ will
pick elements of $E$, say $e_{x},$ with the property $e_{x}=h(x)$ and $%
f\left( e_{x}\right) =g\left( e_{x}\right) .$

Now suppose there is a morphism $u^{\prime }:Y\rightarrow \mathcal{X}%
((C(X),\partial _{X}),\left( A,\alpha \right) )$ such that $\mathcal{U}%
f\circ u^{\prime }=\mathcal{U}g\circ u^{\prime }.$ Thus for any $y\in
Y,u^{\prime }\left( y\right) :(C(X),\partial _{X})\rightarrow \left(
A,\alpha \right) $ will pick $a_{x}=u^{\prime }\left( y\right) \left(
x\right) ,$ for $x\in X,$ and $f\left( a_{x}\right) =g\left( a_{x}\right) .$
So $a_{x}=e_{x}$ for some $e_{x}\in E.$ Define $u^{\prime \prime
}:Y\rightarrow ((C(X),\partial _{X}),\left( E,\varepsilon \right) )$ by $%
u^{\prime \prime }\left( y\right) =h\left( x\right) =e_{x}$ which satisfies
the equality above.

Clearly $\mathcal{U}u\circ u^{\prime \prime }=u^{\prime }.$ Thus $\mathcal{U}%
u$ is the equaliser of $\mathcal{U}f$ and $\mathcal{U}g.$ Hence $\mathcal{U}$
preserves equalisers. Therefore $\mathcal{U}$ is an exact functor.
\end{proof}

Thus, the proof of the Embedding Theorem is completed.

\noindent U. Ege Arslan and G. Onarl\i \newline Department of
Mathematics and Computer Sciences \newline Eski\c{s}ehir Osmangazi
University \newline 26480 Eski\c{s}ehir/Turkey\newline e-mails:
\{uege, gonarli\} @ogu.edu.tr

\end{document}